\documentclass[11pt,a4paper,reqno]{article}
\usepackage{amsmath}
\usepackage{amsthm}
\usepackage{enumerate}
\usepackage{amssymb}

\usepackage{enumerate}
\usepackage{verbatim}
\usepackage{url}
\usepackage[latin1]{inputenc}

\usepackage{caption}
\usepackage{graphicx,color}
\usepackage{tikz}
\usetikzlibrary{calc}
\usepackage{xifthen}
\usetikzlibrary{decorations.pathreplacing}

  \newcounter{constant}

\def\arraypar#1{\parbox[c]{\textwidth - 2cm}{\centering #1}}

\usepackage{environ}
\NewEnviron{display}{
\begin{equation}\begin{array}{c}
\arraypar{\BODY}
\end{array}\end{equation}
}

\usepackage[colorlinks]{hyperref}
\usepackage[figure,table]{hypcap} 
\hypersetup{final}
\usepackage{xcolor}
\hypersetup{
    colorlinks,
    linkcolor={blue!50!black},
    citecolor={blue!50!black},
    urlcolor={blue!80!black}
}

\def\Z{{\mathbb Z}}

\def\R{{\mathbb R}}

\renewcommand*\d{\mathop{}\!\mathrm{d}}

\def\be{\begin{equation}}
\def\ee{\end{equation}}
\def\bea{\begin{equation*}}
\def\eea{\end{equation*}}
\def\bal{\begin{aligned}}
\def\eal{\end{aligned}}
\def\eps{\varepsilon}

\DeclareMathOperator{\cross}{Cross}

\def\Pr{{\mathbb P}}

\newtheorem{theorem}{Theorem}
\newtheorem{lemma}[theorem]{Lemma}

\newtheorem{definition}[theorem]{Definition}

\theoremstyle{remark}
\newtheorem{remark}{Remark}
\newtheorem{preex}[theorem]{Example}

\theoremstyle{definition}

\begin{document}

\title{Existence of an unbounded vacant set for subcritical continuum percolation}
\date{}
\author{Daniel Ahlberg, Vincent Tassion and Augusto Teixeira}
\maketitle

\begin{abstract}
  We consider the Poisson Boolean percolation model in $\mathbb{R}^2$, where the radii of each ball is independently chosen according to some probability measure with finite second moment.
  For this model, we show that the two thresholds, for the existence of an unbounded occupied and an unbounded vacant component, coincide.
  This complements a recent study of the sharpness of the phase transition in Poisson Boolean percolation by the same authors.
  As a corollary it follows that for Poisson Boolean percolation in $\mathbb{R}^d$, for any $d\ge2$, finite moment of order $d$ is both necessary and sufficient for the existence of a nontrivial phase transition for the vacant set.
\end{abstract}

\section*{Introduction}

Percolation theory is the collective name for the study of long-range connections in models of random media. In the most fundamental model of this kind, Bernoulli bond percolation, a discrete random structure is obtained by removing edges of the $\Z^2$ nearest-neighbour lattice by the toss of a coin. The removed edges can be identified with present edges of a slightly shifted `dual' lattice. A milestone within percolation theory was reached in 1980 with Kesten's proof that the thresholds for the existence of an infinite primal and an infinite dual component coincide~\cite{kesten80}. Analogous results have since been obtained for models of percolation in the continuum of $\R^2$, such as Poisson Boolean percolation with bounded radii~\cite{roy90} and Poisson Voronoi percolation~\cite{bolrio06b}. In our previous work \cite{ahltastei}, we characterized the phase transition of the 2 dimensional Poisson Boolean model in terms of crossing probabilities. As a consequence, we proved that the vacant and occupied phase transitions occur at the same parameter under a specific moment assumption (see Equation~\eqref{eq:15} below). As announced in \cite{ahltastei}, we prove in the present paper that the result holds under the minimally required condition that the radii distribution has finite second moment. As a corollary we show that finite moment of order $d$ is both necessary and sufficient for the existence of a nontrivial phase transition for the vacant set in dimension $d\ge2$. We have learned that an alternative argument has been obtained independently by Penrose~\cite{penrose} (submitted June 7, 2017), see Remark~\ref{rem:1} for more details.

In Poisson Boolean percolation, a random occupied set is created based on a Poisson point process in $\R^d$ with intensity parameter $\lambda\ge0$. At each point we center a disc with radius sampled independently from some probability distribution $\mu$ on $\R_+$. The subset $\mathcal{O}\subseteq\R^d$ of points covered by some disc is referred to as the \emph{occupied} set, and its complement $\mathcal{V}=\R^d\setminus\mathcal{O}$ as the \emph{vacant} set. Clearly, the probability of long-range connections in the occupied set is increasing in $\lambda$. Conversely, connections in the vacant set become less likely as $\lambda$ increases. We introduce the threshold parameters
\begin{align*}
\lambda_c&:=\inf\big\{\lambda\ge0:\Pr_\lambda\big[0\stackrel{\mathcal O}{\longleftrightarrow}\infty\big]>0\big\},\\
\lambda_c^\star&:=\sup\big\{\lambda\ge0:\Pr_\lambda\big[0\stackrel{\mathcal V}{\longleftrightarrow}\infty\big]>0\big\},
\end{align*}
where we write $0\stackrel{\mathcal S}{\longleftrightarrow}\infty$ if the connected component  in $\mathcal S$ of the origin is unbounded.

Since a detailed description of the model and an account of previous work on the topic was given by the same authors in the recent paper~\cite{ahltastei}, we shall only recall what is relevant for the results presented here.

Through the paper, when considering Poisson Boolean percolation in dimension $d$, we always assume that the law $\mu$ of the radius size has moment of order $d$:    
\begin{equation}
  \int_0^\infty r^d \,\mu(\d r)<\infty. \label{eq:1}\tag{$\mathbf H_d$}  
\end{equation}
This assumption is very natural since the space $\mathbb R^d$ is entirely occupied for any $\lambda>0$ if \eqref{eq:1} does not hold,  see~\cite{hall85}. In other words, we have for every $\lambda>0$,
\begin{equation}\label{e:cover_space}
 \eqref{eq:1} \iff \Pr_\lambda\big(\mathcal{O}=\mathbb{R}^d\big)=0.
\end{equation}
Furthermore, as soon as \eqref{eq:1} is satisfied, there exists $\lambda>0$ sufficiently small to ensure that the occupied set does not percolate. More precisely,  it is proved in~\cite{gouere08} that this moment assumption is equivalent to the non triviality of the phase transition for the occupied set: 
\begin{equation}
  \eqref{eq:1} \iff \lambda_c\in(0,\infty).\label{eq:20}
\end{equation}
In this paper, we investigate the phase transition of the vacant set. Our main result states that  in dimension $2$, the phase transitions of the vacant and occupied sets occur at the same parameter, under the minimal assumption $(\mathbf H_2)$.

  \begin{theorem}\label{thm:cor}
 Consider Poisson Boolean percolation in $\R^2$, and assume the finite second moment hypothesis~$(\mathbf H_2)$. 
 Then, we have 
 \begin{equation}
   \lambda_c^\star=\lambda_c.\label{eq:4}
 \end{equation}
\end{theorem}
\begin{remark}
  By combining Theorem~\ref{thm:1} below and the results in~\cite{ahltastei}, we obtain a more detailed description of the vacant set in the phases $\lambda > \lambda_c^\star$ and $\lambda = \lambda_c^\star$. More precisely, one can show that Theorem~1.2 in~\cite{ahltastei} holds assuming only the finite second moment condition~$(\mathbf H_2)$.
\end{remark}

In higher dimensions, as a corollary to Theorem~\ref{thm:cor}, we find that finite moment of order $d$ is also sufficient for the existence of an unbounded vacant component at small densities.

\begin{theorem}\label{thm:3}
  Consider Poisson Boolean percolation in $\mathbb{R}^d$, for some $d\ge2$, and assume that \eqref{eq:1} holds. Then,
  \begin{equation}
    \label{eq:16}
    \lambda_c^\star\in(0,\infty).
  \end{equation}
\end{theorem}
Note that $\lambda_c^\star=0$ when \eqref{eq:1} is not satisfied. Hence, the theorem above shows that the phase transition for the vacant set is non trivial if and only if the hypothesis~\eqref{eq:1} is satisfied. Namely, Eq.~\eqref{eq:20} also holds for $\lambda_c^\star$ in place of $\lambda_c$.

As mentioned above, the duality relation \eqref{eq:4} was previously established in \cite{roy90} for bounded radii, and in~\cite{ahltastei} under the assumption
\begin{equation}
\int_0^\infty r^2\log r \,\mu(dr)<\infty.\label{eq:15}
\end{equation}
These previous approaches are based on Russo-Seymour-Welsh techniques and renormalization of crossing probabilities. The hypothesis \eqref{eq:15} implies a decorrelation property of the crossing probabilities of the vacant set, that is sufficient for the standard renormalization techniques to apply. The new step in the proof of Theorem~\ref{thm:cor} is the following theorem, which we prove using specific properties of Boolean percolation valid witout the hypthesis~\eqref{eq:15}, contrary to the standard renormalization 
techniques used in \cite{ahltastei}.
Define $\cross(\ell, r)$ as the event that the box $[0, \ell] \times [0, r]$ can be crossed from left to right by an occupied path.

\begin{theorem}\label{thm:1}
  Consider Poisson Boolean percolation in $\R^2$, and assume the finite second moment hypothesis~$(\mathbf H_2)$. If $\lambda>0$ is such that
  \begin{equation}
    \label{eq:2}
    \lim_{\ell \to\infty}\mathbb P_\lambda[\cross(\ell,3\ell)]=0,
  \end{equation}
  then there exists an unbounded connected vacant component almost surely.
\end{theorem}
\begin{remark}\label{rem:1}
  In a parallel work \cite{penrose}, Penrose obtained a nice proof of the theorem above, using a refined renormalization argument on the crossing probabilities. In the present paper we present a different approach, based on a Peierls-type argument.
\end{remark}

As will be explained in Section~\ref{sec:orga} below, Theorems~\ref{thm:cor} and \ref{thm:3} can easily be derived from Theorem~\ref{thm:1}. The core of the paper is thus devoted to the proof of Theorem~\ref{thm:1}. The proof will follow a Peierls-type argument, but requires some modifications. In order to present the difficulties related to our framework, let us present standard arguments that would prove Theorem~\ref{thm:1} in the radius 1 case (which corresponds to $\mu=\delta_1$). In this case, using a renormalization method based on independence, we can show that \eqref{eq:2} implies the exponential decay of the connection probabilities for the occupied set: there exists a constant $c>0$ such that the probability of an occupied path from a box $\Lambda(x,1)$ around a given point $x$ with radius 1 to distance $r$ around it satisfies
\begin{equation}
  \label{eq:18}
  \mathbb P_\lambda\big[\Lambda(x,1) \stackrel{\mathcal O}{\longleftrightarrow} \Lambda(x,r)^c \big]\le e^{-cr}. 
\end{equation}
To prove that the vacant percolates, one can first observe that   
\begin{equation}
  \label{eq:19}
  \mathbb P_\lambda\big[B(0,L)\ \stackrel{\mathcal V}{\not\!\!\!\longleftrightarrow} \infty\big]= \mathbb P_\lambda \big[\exists \text{ an occupied circuit surrounding $B(0,L)$}\big]
\end{equation}
and then use the exponential decay to show that that the right hand side above tends to $0$ as $L$ tends to infinity. This argument fails when $\mu$ has a fat tail, due long-range dependencies. Indeed, if we choose $\mu$ with a sufficiently fat tail (but still satisfying ($\mathbf H_2$)), the probability that a disc of radius $L$ covers a given point may decay arbitrarily slowly with $L$, and we cannot expect to have a fast decay as in \eqref{eq:18}. Nevertheless, by considering carefully the properties of the combinatorial structures blocking a vacant path (sequences of discs encircling a large box, called ``necklaces'')  one may sill prove that an occupied circuit around the origin is unlikely. A key factor in the calculation will be to discriminate on the radius $r$ of the \emph{second} largest disc in a necklace. Then the rest of the necklace must still contain long paths of discs of size smaller than $r$ and we can quantitatively bound the probability of such events.


The paper is organized as follows. In Section~\ref{sec:orga}, we give the main notation, we derive Theorems~\ref{thm:cor} and \ref{thm:3} from Theorem~\ref{thm:1}, as well as provide some useful lemmas. In Section \ref{sec:proof-theorem}, we complete the proof of Theorem \ref{thm:1}.

\section{Preliminaries}\label{sec:orga}


\paragraph{Notation}
We shall use the notation $B(x,r)$ and $\Lambda(x,r)$ to respectively denote the Euclidean ball and $\ell_\infty$-ball centred at $x\in\R^2$ with radius $r>0$.
Let $\omega$ be a Poisson point process on $\R^2\times\R_+$ of intensity $\lambda\, dx\,\mu(dz)$, for some $\lambda\ge0$. We denote by $\Pr_\lambda$ the law of $\omega$.
Based on $\omega$ we obtain a partition of $\R^2$ into an occupied set $\mathcal{O}$ and vacant set $\mathcal{V}=\R^2\setminus\mathcal{O}$ as follows:
$$
\mathcal{O}(\omega):=\bigcup_{(x,z)\in\omega}B(x,z).
$$

\paragraph{Derivation of Theorems~\ref{thm:cor} and \ref{thm:3} from Theorem \ref{thm:1}}

\begin{proof}[Proof of Theorem~\ref{thm:cor}]
  First, by Corollary 4.5 in \cite{ahltastei}, we have $\lambda_c^\star \le \lambda_c$. To prove the other direction, pick $\lambda<\lambda_c$. By Theorem 1.1 in \cite{ahltastei}, we have $\lim_{\ell \to\infty}\mathbb P_\lambda[\cross(\ell,3\ell)]=0$ which implies $\lambda\le \lambda_c^\star$ by Theorem~\ref{thm:1}, and hence that $\lambda_c^\star\ge \lambda_c$.
\end{proof}

\begin{proof}[Proof of Theorem~\ref{thm:3}]
  We first note that the restriction of the occupied set of Poisson Boolean percolation in $\mathbb{R}^d$ to a two dimensional subspace is in law identical to Poisson Boolean percolation in $\mathbb{R}^2$, now with some modified radius distribution. 
  Instead of calculating how this procedure modifies the radius distribution we can use~\eqref{e:cover_space}, together with the assumption of the corollary, to conclude that the space is not almost surely covered.
  Therefore its restriction to $\mathbb{R}^2$ is also not almost surely covered, so again by~\eqref{e:cover_space} we conclude that condition~($\mathbf H_2$) is satisfied for the induced planar model.
  By Theorem~\ref{thm:cor} we conclude that in the induced model there is almost surely an unbounded vacant component for small enough values of $\lambda$, concluding the proof.
\end{proof}

\paragraph{Connection probabilities for truncated radii}

As explained in the introduction, we cannot have in general quantitative estimates on the decay of connection probabilities, due to the presence of large discs. Nevertheless, if one removes the balls of large radius, we recover the standard exponential decay \eqref{eq:18}, after a suitable rescaling by the length of truncation. This is the content of the next lemma. Let 
$$
\mathcal{O}_\ell(\omega):=\bigcup_{(x,z)\in\omega:\,z\le\ell}B(x,z).
$$

\begin{lemma}
  \label{lem:1}
  Assume that $\lambda\ge0$ is such that~\eqref{eq:2} holds.
  Then there exist constants $\ell_0\ge1$ and $c>0$ such that for every $x\in \mathbb R^2$ and $L\ge\ell \ge \ell_0$, we have
  \begin{equation*}
    \mathbb P_\lambda[\Lambda(x,\ell)\stackrel{\mathcal O_\ell}{\longleftrightarrow} \Lambda(x,L)^c ] \le \frac1c\, e^{-c L/\ell}.
    \label{eq:3}
  \end{equation*}
\end{lemma}

\begin{proof}
  Fix $\eps>0$ and let $\ell_0\ge1$ be such that $\mathbb P_\lambda[\cross(\ell, 3\ell)]<\eps$
  for every $\ell \ge \ell_0$. Define for $\ell\ge \ell_0$ the following
  percolation process on $\mathbb Z^2$: For $y\in \mathbb Z^2$, set $X(y)=1$ if
  there exists a path in $\mathcal{O}_\ell$ from $\Lambda(\ell y,\ell)$ to $\Lambda(\ell y,3\ell)^c$.
  When $X(y)=1$ we say that $y$ is \emph{$X$-open}.
  Then $X$ defines a $10$-dependent percolation process on $\Z^2$, in the sense of~\cite{ligschsta97}, and satisfies $\mathbb P_\lambda[X(y)=1]< 4\eps$.
  Therefore, given that $\eps>0$ is small enough, the probability that there exists an $X$-open nearest-neighbour path\footnote{That is, made up of $X$-open vertices.} of length $k$ in $\Z^2$ from the origin decays exponentially in $k$.

  Observe next that if in $\mathbb R^2$ there exists
  an occupied path from $\Lambda(0,\ell)$ to $\Lambda(0,L)^c$, using only
  balls with radius smaller than $\ell$, then there must exist an $X$-open path in $\Z^2$, starting at the origin, of length
  $\lfloor L/\ell\rfloor$.
  Therefore, for some constant $c>0$,
  $$
  \mathbb P_\lambda[\Lambda(x,\ell)\stackrel{\mathcal O_\ell}{\longleftrightarrow} \Lambda(x,L)^c]\le \frac{1}{c}\,e^{-c\lfloor L/\ell \rfloor},
  $$
  as required.
\end{proof}

A direct consequence of the exponential decay above is the following estimate, that will be useful to apply the Peierls argument. Let $E_\ell(L)$ be the event that for some $x\in \R^2$ satisfying $|x|\ge L$ there is a path in $\mathcal{O}_\ell$ from $x$ to $B(x,|x|)^c$.

\begin{lemma} \label{lem:2}
  Assume that $\lambda\ge0$ is such that~\eqref{eq:2} holds.
  Then there exist constants $\ell_0\ge1$ and
  $c>0$ such that for every $L\ge\ell \ge \ell_0$ we have
  \begin{equation*}
    \mathbb P_\lambda[E_\ell(L)]\le \frac1c\,  e^{-c L/\ell}.
  \end{equation*}
\end{lemma}

\begin{proof}
  If $E_\ell(L)$ occurs, then there must exist a point $y\in \ell\mathbb Z^2\setminus B(0,L-\ell)$ such that $\Lambda(y,\ell)$ is connected to distance $|y|-\ell$ around it. By the exponential decay of Lemma~\ref{lem:1}, we obtain
  \begin{equation}
    \label{eq:21}
    \mathbb P_\lambda[E_\ell(L)] \le \sum_{y\in \ell\mathbb Z^2\setminus B(0,L-\ell)} \frac1c e^{-c |y|/\ell}\le \frac1{c'}e^{-c'L/\ell}
  \end{equation}
for some $c'>0$.
\end{proof}

\paragraph{Large discs are far from the origin}
We next estimate the probability of seeing two large discs close to the origin. Given $r,s>0$ let $F(r,s)$ be the event that there exist at least two discs $B_1, B_2$
such that for $i=1,2$,
\begin{enumerate}[(i)]
\item\label{item:1} the radius of $B_i$ satisfies $\mathrm{rad}(B_i)\ge r$, and
\item\label{item:2} the Euclidean distance between $0$ and $B_i$ satisfies $0<d(0,B_i)\le s$.
\end{enumerate}

\begin{lemma}\label{lem:3}
  For every $\lambda\ge0$ we have that
  \begin{equation*}
    \mathbb P_\lambda[F(r,s)]  \,\le\,
    \lambda^2 \bigg (\pi s^2\mu\left([r,\infty)\right)+2\pi s\int_r^\infty z \,\mu(dz) \bigg)^2.
  \end{equation*}
\end{lemma}

\begin{proof}
  The number of open balls satisfying items~(\ref{item:1}) and~(\ref{item:2}) above is a
  Poisson random variable with parameter
  \begin{align*}
    &\lambda\cdot\mathrm{Leb}\otimes\mu \left[\left\{(x,z)\in \mathbb  R^2 \times \mathbb
      R^+: r\le z<|x|\le s+z \right\}\right]\\
    &\qquad=\,\lambda\int_r^\infty
      \bigg(2\pi\int_z^{s+z}\rho\,d\rho\bigg) \mu(dz)\\
    &\qquad=\,\lambda\bigg(\pi s^2 \mu([r,\infty))+2\pi s\int_r^\infty z\, \mu(dz)\bigg).
  \end{align*}
  The result then follows since a Poisson variable with parameter $\nu$ has probability at most $\nu^k$ of being larger than or equal to $k$.
\end{proof}

\section{Proof of Theorem~\ref{thm:1}}
\label{sec:proof-theorem}

Throughout this section we shall assume  ($\mathbf H_2$) and that  $\lambda > 0$ is such that~\eqref{eq:2} holds. The proof of Theorem~\ref{thm:1} will proceed counting necklaces, which we define as follows.

\begin{definition}
  A sequence $(B_1,B_2,\ldots,B_k)$ of balls of decreasing size, i.e.\ where $\mathrm{rad}(B_1)\ge\mathrm{rad}(B_2)\ge\cdots\ge \mathrm{rad}(B_k)$, will be called a {\bf necklace around $B(0,L)$} if
  \begin{enumerate}[(i)]
  \item the complement of $\bigcup_{i\in[k]}B_i$ consists of two connected components, one bounded and one unbounded;
  \item $B(0,L)$ is contained in the bounded component;
  \item the complement of $\bigcup_{i\in[k]\setminus\{j\}}B_i$ consists of a unique connected component for every $j=1,2,\ldots,k$.
  \end{enumerate}
\end{definition}
\begin{remark}
  Alternatively, we could have defined simply a necklace as a connected set of balls disjoint from $B(0,L)$ and containing an occupied circuit around the origin. The proof would follow the steps, up to minor modifications.
\end{remark}
Notice, in particular, that if $B_1,B_2,\ldots,B_k$ is a necklace around $B(0,L)$, then none of the balls intersect $B(0,L)$.  We shall proceed via a Peierls-type of argument, and show that for $L$ large enough,  the
probability that $B(0,L)$ intersect an unbounded vacant component is positive. We begin with the  dual reformulation of the problem  
\begin{align}
  \label{eq:17}
  \mathbb P[B(0,L)\ \stackrel{\mathcal V}{\not\!\!\!\longleftrightarrow} \infty\big]&= \mathbb P_\lambda \big[\exists \text{ an occupied circuit surrounding $B(0,L)$}\big]\notag\\ 
&= \mathbb  P_\lambda\big[\exists\text{ necklace around $B(0,L)$}\big].
\end{align}
The second equality is true because the number of balls that intersect a bounded region is almost surely finite (see e.g.~\cite[(2.15)]{ahltastei}). Indeed, if there exists an occupied circuit surrounding $B(0,L)$ we can assume that this circuits is included in the union of finitely many discs. We construct a necklace around $B(0,L)$ by removing some discs until (i) and (iii) are satisfied.

In order to bound the probability of a necklace around $B(0,L)$, we discriminate on the size of the second largest pearl. Fix $L\ge0$. For $0\le a\le b$, let $G_L(a,b)$ be the event that there exists a necklace $(B_1,B_2,\ldots, B_k)$ around $B(0,L)$ with $\mathrm{rad}(B_2)\in[a,b]$. Let $j_0:=\lfloor\log_3(\sqrt{L})\rfloor$. Applying the union bound to \eqref{eq:17}, we get
\begin{equation}
 \mathbb P[B(0,L)\ \stackrel{\mathcal V}{\not\!\!\!\longleftrightarrow} \infty\big] \,\le\, \Pr_\lambda[G_L(0,\sqrt{L})] + \sum_{j\ge j_0} \Pr_\lambda[G_L(3^j,3^{j+1})].\label{eq:6}
\end{equation}
We bound the terms on the right-hand side, by using the two lemmas below.

\begin{lemma}
  There exists a constant $c>0$ such that for all large $L$ we have
  \begin{equation}
    \label{eq:7}
    \mathbb P_\lambda[G_L(0,\sqrt L)]\, \le\, \frac{1}{c}\,e^{-c\sqrt L}.
  \end{equation}
\end{lemma}

\begin{proof}
  Assume that there exists a necklace $(B_1,\ldots,B_k)$ around
  $B(0,L)$ such that $\mathrm{rad}(B_2)\le\sqrt L$. Consider a disc
  $B_j$ of the necklace that intersects $B_1$. Starting from the
  center $x_j$ of $B_j$ there must exist an open path of length at least
  $|x_j|$ made of balls with radius smaller than $\sqrt L$, see Figure~\ref{fig:necklace}. Hence the event
  $E_{\sqrt L}(L)$ occurs. So, by Lemma~\ref{lem:2}, we have
  \begin{equation*}
    \mathbb P_\lambda[G_L(0,\sqrt L)] \,\le\,  \mathbb P_\lambda[E_{\sqrt L}(L)]\, \le\, \frac{1}{c}\,e^{-c\sqrt L},
  \end{equation*}
  for some $c>0$ and all large $L$.
\end{proof}

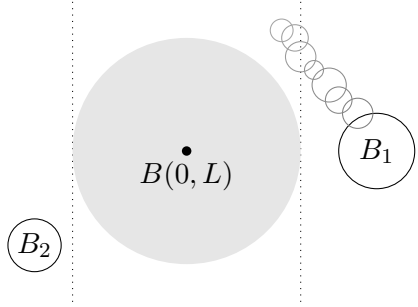
\begin{figure}[htbp]
  \begin{center}
    \begin{tikzpicture}[scale=.5]
      \def\r{3};
      \draw[dotted] (-\r,-4) -- (-\r,4);
      \draw[dotted] (\r,-4) -- (\r,4);
      \fill[fill=gray!20!white] (0,0) circle [radius=\r];
      \draw (5,0) circle [radius=1]; 
      \draw (-4,-2.5) circle [radius=.7]; 
      \draw[gray] (4.5,1) circle [radius=.4];
      \draw[gray!95!white] (4,1.35) circle [radius=.35];
      \draw[gray!90!white] (3.75,1.75) circle [radius=.45];
      \draw[gray!85!white] (3.35,2.15) circle [radius=.25];
      \draw[gray!80!white] (3,2.5) circle [radius=.4];
      \draw[gray!75!white] (2.85,3) circle [radius=.35];
      \draw[gray!70!white] (2.5,3.2) circle [radius=.3];
      \fill[fill=black] (0,0) circle (.125);
      \draw (5,0) node {$B_1$};
      \draw (-4,-2.5) node {$B_2$};
      \draw (0,0) node[anchor=north] {$B(0,L)$};
    \end{tikzpicture}
    \caption{On the event $G_L(a,b)$ there exist two sequences of discs with radii at most $b$ connecting $B_1$ to $B_2$, while avoiding the shaded region. At least one of the two must have length at least $L$.
    }
    \label{fig:necklace}
  \end{center}
\end{figure}

Define $p(r)=\pi r^2 \mu\left([r,\infty)\right)+2\pi r\int_r^\infty z\, \mu(dz)$.

\begin{lemma}
  There exists a constant $c>0$ such that for every $L,r\ge1$ we have
  \begin{equation}
    \label{eq:8}
    \mathbb P_\lambda[G_L(r,3r)]\, \le\, \frac{\lambda}{c}\, p(r).
  \end{equation}
\end{lemma}

\begin{proof}
  Let $a\ge 1$ be a constant to be chosen later.
  First, by Lemma~\ref{lem:3},
  \begin{equation}\label{eq:13}
    \begin{aligned}
      \mathbb P_\lambda[G_L(r,3r)]\,&\le\, \mathbb P_\lambda[F(r, ar)]+ \mathbb
      P_\lambda[G_L(r,3r) \setminus F(r,ar)]\\
      &\le\, \lambda^2 a^4 p(r)^2 +\mathbb P_\lambda[G_L(r,3r) \setminus F(r,ar)].
    \end{aligned}
  \end{equation}
  Assume that $G_L(r,3r)$ occurs but not $F(r,ar)$. Let $(B_1,\ldots,B_k)$ be a
  necklace around $B(0,L)$ such that $r\le\mathrm{rad}(B_2)\le3 r$. Since the
  discs $B_1$ and $B_2$ have radii larger than $r$, but $F(r,ar)$ does not occur, at least one of them
  must be at distance at least $ar$ from $0$.
  By considering a ball of the necklace intersecting $B_1$, one can see
  as above that the event $E_{3r}(ar)$ must occur. Hence,
  Lemma~\ref{lem:2} gives
  \begin{equation}
     \label{eq:14}
    \mathbb P_\lambda[G_L(r,3r) \setminus F(r,ar)]\,\le\,   \mathbb
    P_\lambda[E_{3r}(ar)]
    \,\le\, \frac1c\,e^{-ca/3}
    \,\le\, \frac{1}{c'a^4}.
  \end{equation}
  For $a=(\lambda p(r))^{-1/4}$ Equations~\eqref{eq:13} and~\eqref{eq:14} together give~\eqref{eq:8}.
\end{proof}

We are now ready to complete the proof of Theorem~\ref{thm:1}.

\begin{proof}[Proof of Theorem~\ref{thm:1}]
  Recall first that $j_0 = \lfloor \log_3 (\sqrt{L}) \rfloor$.
  Combining~\eqref{eq:6},~\eqref{eq:7} and~\eqref{eq:8} we find a constant $c>0$ such that for all large $L$ we have
  \begin{equation*}
    \mathbb  P_\lambda\big[\exists\text{ necklace around $B(0,L)$}\big]
    \,\le\, \frac1c\,e^{-c\sqrt L}+ \frac{\lambda}{c} \sum_{j\ge j_0} p(3^j).
  \end{equation*}
  Using Fubini's theorem we may interchange the order of summation, and obtain for $z_0=3^{j_0}\ge\sqrt L/3$ the following upper bound on the infinite sum:
  \begin{equation*}
    \begin{aligned}
      \sum_{j\ge j_0}p(3^j)\,&=\,\pi\sum_{j\ge j_0}3^{2j}\int_{3^j}^\infty\mu(dz)+2\pi\sum_{j\ge j_0}3^j\int_{3^j}^\infty z\,\mu(dz)\\
      &\le\,\pi\int_{z_0}^\infty 81z^2\,\mu(dz)+2\pi\int_{z_0}^\infty 9z^2\,\mu(dz)\\
      &\le\,99\pi\int_{z_0}^\infty z^2\,\mu(dz),
    \end{aligned}
  \end{equation*}
  which tends to zero as $L$ increases, due to the assumption in~\eqref{eq:1}.
 This completes the proof of Theorem~\ref{thm:1}.
\end{proof}

\bibliographystyle{plain}
\bibliography{bib}

\end{document}